\documentclass{article}

\usepackage{latexsym,amstext,amsmath,amssymb,amsthm,bbm}
\usepackage[all]{xy}

\title{Filtered restriction of geometric $\D$-modules}

\author{R\'emi Arcadias\\
Depto.\ de \'Algebra, Fac.\ Matem\'aticas, Campus Reina Mercedes\\
41012 Sevilla (Spain)\\
E-mail: rarcadias@us.es}

\newtheorem{proposition}{Proposition}[section]

\newtheorem{lemma}{Lemma}[section]
\newtheorem{corollary}{Corollary}[section]

\theoremstyle{definition}
\theoremstyle{definition}\newtheorem{definition}{Definition}[section]
\theoremstyle{definition}
\theoremstyle{definition}\newtheorem{remark}{Remark}[section]

\newcommand{\D}{\mathcal{D}}
\newcommand{\C}{\mathbb{C}}

\begin{document}

\maketitle



\section*{Introduction}

Let $\D$ (resp.\ $\D'$) be the ring of germs at the origin of linear partial differential operators with analytic coefficients on $\C^n$ (resp.\ $\C^n\times \C^p$). We take a system of coordinates $(x_1,\dots,x_n)$ on $\C^n$ and  $(x_1,\dots,x_n,t_1,\dots,t_p)$  on $\C^n\times\C^p$ and we denote by $i:\C^n\to\C^n\times\C^p$ the inclusion $x\to (x,0)$. 
Let $M$ be a finitely generated $\D'$-module specializable along 
$Y=\{t_1=\cdots=t_p=0\}$. The $\D$-module theoretic restriction $i^*M$ of $M$ along $Y$ is defined as the complex $(\D'/\sum t_i\D')\otimes^L M$, realized as a Koszul complex $(M\otimes \Lambda^i \C^p)_i$, whose cohomology groups are finitely generated over $\D$.
T.\ Oaku and N.\ Takayama show in \cite{OT01b} how to realize this complex by a complex composed of finitely generated $\D$-modules. Let $V(\D')$ denote the $V$-filtration of Malgrange-Kashiwara along $Y$, and let $k_1$ be the maximal integral root of the $b$-function of $M$ along $Y$. Then $(M\otimes \Lambda^i \C^p)_i$ is quasi-isomorphic to the sub-complex  $(V_{k_1+i}(M)\otimes \Lambda^i \C^p)_i$.
Then a free resolution of $i^*M$ can be computed as follows:
let
\begin{equation}\label{eq*}
\cdots\to\mathcal{L}_1\to\mathcal{L}_0\to M\to 0
\end{equation}
be a free resolution of $M$ adapted to the $V$-filtration. Then $(V_{k_1+i}(M)\otimes \Lambda^i \C^p)_i$ is isomorphic, in the derived category of $\D$-modules, to the complex
\[
\cdots\to\frac{V_{k_1}(\mathcal{L}_1)}{\sum t_i V_{k_1+1}(\mathcal{L}_1)}\to
\frac{V_{k_1}(\mathcal{L}_0)}{\sum t_iV_{k_1+1}(\mathcal{L}_0)}\to 0,
\]
denoted by $(V_{k_1}(\mathcal{L}_i)/(\sum_j t_jV_{k_1+1}(\mathcal{L}_{i+1})))_i$.

On the other hand, the most classical filtration in $\D$-module theory is that by the order, denoted by $(F_d(\D))$. One defines the notion of a minimal filtered free resolution of a $F$-filtered $\D$-module, thus one has the notion of Betti numbers of such a module. We remark that the notion of Betti numbers can be defined as well for complexes of $\D$-modules in the derived category of $F$-filtered $\D$-modules $DF(\D)$. Furthermore, T. Oaku and N. Takayama\cite{OT01} and M. Granger and T. Oaku\cite{GO04} define the notion of a minimal bifiltered free resolution of a $(F,V)$-bifiltered $\D'$-module. With the help of this theory, we establish 
 that the isomorphism 
 \[(V_{k_1+i}(M)\otimes \Lambda^i \C^p)_i\simeq
(V_{k_1}(\mathcal{L}_i)/(\sum_j t_jV_{k_1+1}(\mathcal{L}_{i+1})))_i
\]
can be seen in the category $DF(\D)$ (Proposition \ref{prop9}). 

A further property on complexes of $F$-filtered $\D$-modules is that of strictness, a property  any filtered free resolution of a $F$-filtered module satisfies. In the case where $p=1$ and $t:M\to M$ is injective, we give conditions such that the complex 
\[
(V_{k_1+i}(M)\otimes \Lambda^i \C^p)_i=(0\to V_{k_1+1}(M)\stackrel{t}{\to} V_{k_1}(M)\to 0)
\]
is strict (Proposition \ref{prop10}). In that case the complex 
$$
(V_{k_1}(\mathcal{L}_i)/(t V_{k_1+1}(\mathcal{L}_{i+1})))_i
$$ 
becomes a $F$-filtered free resolution of the module $M/tM$. This latter fact, suggested by Toshinori Oaku, was the original motivation of this paper.

 In the last section, we apply our results to the algebraic local cohomology module 
 $N=\mathcal{O}[1/f]/\mathcal{O}$, seen as a $\D$-module. Here $f$ is a quasi-homogeneous polynomial with an isolated singularity at the origin.
$N$ is endowed with the good $F$-filtration
\[
F_d(N)=\sum_{1\leq j\leq k}F_{d-j-1}(\D)[1/f^j],
\]
which takes into account the order of the pole $f$, where $-k$ is the least integral root of the Bernstein-Sato polynomial associated with $f$. By using the restriction, we give a minimal presentation associated with this data (Proposition \ref{prop4}). 

\section{Free resolutions of $\D$-modules}

Let $\D=\C\{x_1,\dots,x_n\}[\partial_1,\dots,\partial_n]$ denote the ring of germs at the origin of linear partial differential operators with analytic coefficients on $\C^n$. It is endowed with a filtration by the order, denoted by $(F_d(\D))_{d\in\mathbb{Z}}$. 

Let $M$ be a finitely generated left module over $\D$. An $F$-filtration $(F_d(M))_{d\in\mathbb{Z}}$ on $M$ is an exhaustive sequence of subspaces 
satisfying  $F_d(\D)F_{d'}(M)\subset F_{d+d'}(M)$ for any $d,d'$. It is called a good $F$-filtration if moreover there exist $f_1,\dots,f_r\in M$ and a vector shift $\mathbf{n}=(n_1,\dots,n_r)$ such that for any $d$,
\[
F_{d}(M)=\sum F_{d-n_i}(\D)f_i.
\]
For example, we denote by $\D^r[\mathbf{n}]$ the free module $\D^r$ with basis $e_1,\dots,e_r$ endowed with the filtration 
\[
F_{d}(\D^r[\mathbf{n}])=\sum F_{d-n_i}(\D)e_i.
\]
Thus a good $F$-filtration of $M$ is defined by a surjective map 
$\D^r[\mathbf{n}]\to M$. 

One then defines the notion of a $F$-filtered free resolution: it is an exact sequence
\begin{equation}\label{eq0}
\cdots\to\mathcal{L}_1\to\mathcal{L}_0\to M\to 0
\end{equation}
which induces for each $d$ an exact sequence of vector spaces
\[
\cdots\to F_d(\mathcal{L}_1)\to F_d(\mathcal{L}_0)\to F_d(M)\to 0.
\]
Following T. Oaku and N. Takayama\cite{OT01} and M. Granger and T. Oaku\cite{GO04}, we are able to define the notion of a \emph{minimal} 
$F$-filtered free resolution. To that end we introduce, following F.J.\ Castro-Jim\'enez and L.\ Narv\'aez-Macarro\cite{narvaez97}, the homogenization ring $\mathbf{R}\D=\bigoplus F_d(\D)T^d$. It is isomorphic to the ring 
\[
\D^{(h)}=\C\{\mathbf{x}\}[\partial_1,\dots,\partial_n,h]
\]
satisfying for any $i$, $\partial_ix_i-x_i\partial_i=h$. The ring $\D^{(h)}$ is graded by the order in $\partial,h$. The homogenization of a $F$-filtered module $M$ is defined by 
\[
\mathbf{R}M=\bigoplus_d F_d(M)T^d. 
\]
In particular we have an isomorphism 
\[
\mathbf{R}(\D^r[\mathbf{n}])\simeq(\D^{(h)})^r[\mathbf{n}]
\]
where at the right hand side the vector shift $\mathbf{n}$ refers to the grading of the module. The $F$-filtered free resolution (\ref{eq0}) induces a graded free resolution 
\begin{equation}\label{eq0b}
\cdots\to\mathbf{R}\mathcal{L}_1\to\mathbf{R}\mathcal{L}_0\to \mathbf{R}M\to 0.
\end{equation}
This resolution is called \emph{minimal} if all the entries of the matrices representing the maps $\mathbf{R}\mathcal{L}_i\to\mathbf{R}\mathcal{L}_{i-1}$ belong
to the maximal two-sided ideal of $\D^{(h)}$ generated by $(x_i)_i,(\partial_i)_i,h$.
The resolution (\ref{eq0}) is called minimal if it induces a minimal graded free resolution (\ref{eq0b}). There exists a minimal $F$-filtered free resolution of $M$, unique up to $F$-filtered isomorphism (see \cite{GO04}).
 
Let us denote now 
$\D_{x,t}=\C\{x_1,\dots,x_n,t_1,\dots,t_p\}
[\partial_1,\dots,\partial_n,\partial_{t_1},\dots,\partial_{t_p}]$ the ring of germs of linear partial differential operators with analytic coefficients on $\C^n\times \C^p$. It is endowed with the $F$-filtration by the order $(F_d(\D_{x,t}))$. Another filtration is the so-called $V$-filtration of B.\ Malgrange and M.\ Kashiwara, denoted by $(V_k(\D_{x,t}))_{k\in\mathbb{Z}}$, defined as follows: 
let us define the $V$-order of a monomial 
\[
\mathrm{ord}^V(\lambda\mathbf{x}^{\mathbf{\alpha}}t^{\mu} 
\partial^{\mathbf{\beta}}\partial_t^{\nu})=\sum \nu_i-\sum \mu_i,
\]
with $\lambda\in\C\setminus 0$,
then $V_k(\D_{x,t})$ denote the set of operators whose development only contains monomials $\mathbf{x}^{\mathbf{\alpha}}t^{\nu} 
\partial^{\mathbf{\beta}}\partial_t^{\mu}$ having order at most $k$. 

We have also the notion of a 
$V$-filtration of a $\D_{x,t}$-module: that is an exhaustive sequence of subspaces $(V_k(M))_{k\in\mathbb{Z}}$ such that for any $k,k'$, we have 
$V_k(\D_{x,t})V_{k'}(M)\subset V_{k+k'}(M)$. As above let
 $\D_{x,t}^r[\mathbf{m}]$ denote the free module $\D_{x,t}^r$ with basis $e_1,\dots,e_r$ endowed with the $V$-filtration 
\[
V_{k}(\D_{x,t}^r[\mathbf{m}])=\sum V_{k-m_i}(\D_{x,t})e_i.
\]
If $M$ is endowed with a $V$-filtration defined from a surjective map $\D_{x,t}^r[\mathbf{m}]\to M$, then $M$ admits a $V$-filtered free resolution: that is an exact sequence
\begin{equation*}
\cdots\to\D_{x,t}^{r_1}[\mathbf{m}^{(1)}]\to
\D_{x,t}^{r_0}[\mathbf{m}^{(0)}]\to M\to 0
\end{equation*}
which induces for each $k$ an exact sequence of vector spaces 
\begin{equation*}
\cdots\to V_k(\D_{x,t}^{r_1}[\mathbf{m}^{(1)}])\to
V_k(\D_{x,t}^{r_0}[\mathbf{m}^{(0)}])\to V_k(M)\to 0.
\end{equation*}
A $V$-filtered free resolution is the object needed in the computation of the restriction, as we shall see in the next section. 

Let us describe now how to define the notion of a minimal $V$-filtered free resolution of $M$. To that end we introduce the bifiltration
\[
F_{d,k}(\D_{x,t})=F_d(\D_{x,t})\cap V_k(\D_{x,t})
\]
for each $d,k\in\mathbb{Z}$. A bifiltration of a $\D_{x,t}$-module $M$ 
is an exhaustive sequence of subspaces $(F_{d,k}(M))$
satisfying  $F_{d,k}(M)\subset F_{d+1,k}(M)\cap F_{d,k+1}(M)$ and $F_{d,k}(\D_{x,t})F_{d',k'}(M)\subset F_{d+d',k+k'}(M)$ for any $d,d',k,k'$. It is called a good bifiltration if moreover there exist $f_1,\dots,f_r\in M$ and two vector shifts $\mathbf{n}=(n_1,\dots,n_r)$ and $\mathbf{m}=(m_1,\dots,m_r)$ such that for any $d,k$,
\[
F_{d,k}(M)=\sum F_{d-n_i,k-m_i}(\D_{x,t})f_i.
\]
For example, we denote by $\D^r[\mathbf{n}][\mathbf{m}]$ the free module $\D_{x,t}^r$ with basis $e_1,\dots,e_r$ endowed with the good bifiltration 
\[
F_{d,k}(\D_{x,t}^r[\mathbf{n}][\mathbf{m}])=\sum F_{d-n_i,k-m_i}(\D_{x,t})e_i.
\]
Let us point out that if $M$ is endowed with a good bifiltration $(F_{d,k}(M))$, then it is endowed with a good $F$-filtration $F_d(M)=\cup_k F_{d,k}(M)$ and a $V$-filtration $V_k(M)=\cup_d F_{d,k}(M)$.
A bifiltered free resolution of $M$ is an exact sequence
\begin{equation*}
\cdots\to\D_{x,t}^{r_1}[\mathbf{n}^{(1)}][\mathbf{m}^{(1)}]\to
\D_{x,t}^{r_0}[\mathbf{n}^{(0)}][\mathbf{m}^{(0)}]\to M\to 0
\end{equation*}
which induces for each $d,k$ an exact sequence of vector spaces 
\begin{equation*}
\cdots\to F_{d,k}(\D_{x,t}^{r_1}[\mathbf{n}^{(1)}][\mathbf{m}^{(1)}])\to
F_{d,k}(\D_{x,t}^{r_0}[\mathbf{n}^{(0)}][\mathbf{m}^{(0)}])\to F_{d,k}(M)\to 0.
\end{equation*}
The ring $\D_{x,t}^{(h)}=\mathbf{R}\D_{x,t}$ is endowed by a $V$-filtration $V_k(\D_{x,t}^{(h)})$ similarly as above, by giving the weight $-1$ to each $t_i$ and the weight $1$ to each $\partial_{t_i}$. Equivalently, we have $V_k(\D_{x,t}^{(h)})=\bigoplus_d V_k(\D_{x,t})T^d$. In the same way we endow $\mathbf{R}M$ with the $V$-filtration $V_k(\mathbf{R}M)=\bigoplus_d F_{d,k}(M)T^d$. For example, we have an isomorphism  of $V$-filtered graded $\D_{x,t}^{(h)}$-modules 
\[
\mathbf{R}(\D_{x,t}^{(r)}[\mathbf{n}][\mathbf{m}])\simeq (\D_{x,t}^{(h)})^r[\mathbf{n}][\mathbf{m}]
\] 
where the vector shift $[\mathbf{n}]$ (resp.\ $[\mathbf{m}]$) refers to the grading  (resp.\ $V$-filtration).

Let us take a bifiltered free resolution 
\begin{equation}\label{eq0c}
\cdots\to\mathcal{L}_1\to\mathcal{L}_0\to M\to 0.
\end{equation}
It induces a $V$-filtered graded free resolution
\[
\cdots\to\mathbf{R}\mathcal{L}_1\to\mathbf{R}\mathcal{L}_0\to
\mathbf{R}M\to 0,
\]
thus a bigraded free resolution 
\begin{equation}\label{eq0d}
\cdots\to\mathrm{gr}^V(\mathbf{R}\mathcal{L}_1)\to
\mathrm{gr}^V(\mathbf{R}\mathcal{L}_0)\to
\mathrm{gr}^V(\mathbf{R}M)\to 0.
\end{equation}
Since it makes sense to talk about a minimal bigraded free resolution as (\ref{eq0c}), the authors of \cite{GO04} adopt the following definition: a bifiltered free resolution (\ref{eq0c}) is said to be minimal if so is the bigraded free resolution (\ref{eq0d}). They prove the existence and uniqueness (up to bifiltered isomorphism) of such a resolution.

\section{Free resolutions of complexes of $\D$-modules}

In that section we intend to define the notion of a minimal filtered free resolution of a complex of $\D$-modules in the corresponding derived category. We will review some facts,  perhaps well-known to specialists, for the clarity of our text. 

Let us denote by $CF(\D)$ the category of bounded complexes $\cdots\to M_i\to M_{i-1}\to\cdots$ where for each $i$, $M_i$ is a $F$-filtered $\D$-module and the differentials are $F$-adapted. 

A map $\alpha:M_{\bullet}\to N_{\bullet}$ is said to be a filtered quasi-isomorphism if the induced map $\mathrm{gr}({\alpha}):\mathrm{gr}^FM_{\bullet}\to \mathrm{gr}^FN_{\bullet}$ is a quasi-isomorphism. 

If $\alpha:M_{\bullet}\to N_{\bullet}$ is a morphism in $CF(\D)$, we denote by $C(\alpha)$ the mapping cone of $\alpha$, defined by $C(\alpha)_n=M_{i-1}\oplus N_i$ with differential $(x,y)\mapsto (-\phi(x),\psi(y)+\alpha(x))$. It is endowed with the filtration $F_d(C(\alpha)_i)=F_d(M_{i-1})\oplus F_d(N_i)$. In that way we have natural isomorphisms 
$F_d(C(\alpha))\simeq C(F_d(\alpha))$ and 
$\mathrm{gr}^F(C(\alpha))\simeq C(\mathrm{gr}^F(\alpha))$.

\begin{lemma}
Let $\alpha:M_{\bullet}\to N_{\bullet}$ as above. If for any $d$, $F_d(\alpha)$ is a quasi-isomorphism, then $\mathrm{gr}^F(\alpha)$ is a quasi-isomorphism.
\end{lemma}

\begin{proof}
$F_d(\alpha)$ is a quasi-isomorphism is equivalent to saying that $C(F_d(\alpha))$ is exact. Thus $F_d(C(\alpha))$ is exact and it easily implies that $\mathrm{gr}^F(C(\alpha))\simeq C(\mathrm{gr}^F(\alpha))$ is exact. That gives the result.
\end{proof} 

We have a converse statement:

\begin{proposition}\label{prop7}
Let $\alpha:M_{\bullet}\to N_{\bullet}$ be a morphism in $CF(\D)$, such that for each $i$, $M_i$ and $N_i$ are endowed with good filtrations. If $\mathrm{gr}^F(\alpha)$ is a quasi-isomorphism, then $\alpha$ is a quasi-isomorphism and induces for each $d$ a quasi-isomorphism $F_d(\alpha)$.
\end{proposition}

\begin{proof}
We have that $\mathrm{gr}^F(C(\alpha))\simeq C(\mathrm{gr}^F(\alpha))$ is exact. 
Since each $C(\alpha)_i$ is endowed with a good filtration, then by a statement analogous to \cite{GO04}, Proposition 2.5, it follows that $C(\alpha)$ is exact, moreover for each $d$, $F_d(C(\alpha))$ is exact, which establishes the statement.
\end{proof}

\begin{corollary}
Let $\alpha:M_{\bullet}\to N_{\bullet}$ be a morphism in $CF(\D)$, such that for each $i$, $M_i$ and $N_i$ are endowed with good filtrations. Then $\alpha$ is a filtered quasi-isomorphism if and only if for each $d$, $F_d(\alpha)$ is a quasi-isomorphism. Moreover, if $\alpha$ is a filtered quasi-isomorphism, then it is a quasi-isomorphism.
\end{corollary}

A complex 
\[
\cdots\to M_i\stackrel{\phi_i}{\to} M_{i-1}\to\cdots
\]
in $CF(\D)$ 
is said to be \emph{strict} if for each $i$ and $d$ we have 
$\mathrm{Im}\phi_i\cap F_d(M_{i-1})=\phi_i(F_d(M_{i}))$.  

\begin{lemma}
Let 
\[
\xymatrix{
M_2\ar[r]^{\phi_2}\ar[d] & M_1\ar[r]^{\phi_1}\ar[d]^{\alpha_1} & M_0\ar[d]\\
N_2\ar[r]^{\psi_2} & N_1\ar[r]^{\psi_1} & N_0
}
\]
be a commutative diagram of filtered $\D$-modules, such that the rows are complexes and 
$\alpha_1$ induces isomorphisms $H_1(M_{\bullet}))\simeq H_1(N_{\bullet}))$ and 
$H_1(F_d(M_{\bullet})))\simeq H_1(F_d(N_{\bullet})))$.
Then $\alpha_1$ induces an isomorphism
\[
\frac{F_d(\phi_{2}(M_{2}))}{\phi_{2}(F_d(M_{2}))}
\simeq
\frac{F_d(\psi_{2}(N_{2}))}{\psi_{2}(F_d(N_{2}))}.
\]
\end{lemma}

\begin{proof}
We have a commutative diagram with exact rows
\[
\xymatrix{
\frac{F_d(\mathrm{Ker}\,\phi_1)}{\phi_2(F_d(M_2))}\ar[r]\ar[d]^{\mathrm{iso}} &
\frac{\mathrm{Ker}\,\phi_1}{\mathrm{Im}\,\phi_2}\ar[r]\ar[d]^{\mathrm{iso}} &
\frac{\mathrm{Ker}\,\phi_1}{F_d(\mathrm{Ker}\,\phi_1)+
\mathrm{Im}\,\phi_2}\ar[r]\ar[d] &
0\\
\frac{F_d(\mathrm{Ker}\,\psi_1)}{\psi_2(F_d(N_2))}\ar[r] &
\frac{\mathrm{Ker}\,\psi_1}{\mathrm{Im}\,\psi_2}\ar[r] &
\frac{\mathrm{Ker}\,\psi_1}{F_d(\mathrm{Ker}\,\psi_1)+
\mathrm{Im}\,\psi_2}\ar[r] &
0,
}
\]
then by the five lemma, $\alpha_1$ induces an isomorphism 
\[
\frac{\mathrm{Ker}\,\phi_1}{F_d(\mathrm{Ker}\,\phi_1)+\mathrm{Im}\,\phi_2}
\simeq
\frac{\mathrm{Ker}\,\psi_1}{F_d(\mathrm{Ker}\,\psi_1)+\mathrm{Im}\,\psi_2}.
\]
Similarly, from the diagram with exact rows 
\[
\xymatrix{
0 \ar[r] & \frac{F_d(\mathrm{Ker}\,\phi_1)}{F_d(\mathrm{Im}\,\phi_2)}
\ar[r]\ar[d] &
\frac{\mathrm{Ker}\,\phi_1}{\mathrm{Im}\,\phi_2}\ar[r]\ar[d]^{\mathrm{iso}}  & 
\frac{\mathrm{Ker}\,\phi_1}{F_d(\mathrm{Ker}\,\phi_1)+\mathrm{Im}\,\phi_2}\ar[r]
\ar[d]^{\mathrm{iso}}  &
0\\
0 \ar[r] & \frac{F_d(\mathrm{Ker}\,\psi_1)}{F_d(\mathrm{Im}\,\psi_2)}\ar[r] &
\frac{\mathrm{Ker}\,\psi_1}{\mathrm{Im}\,\psi_2}\ar[r]& 
\frac{\mathrm{Ker}\,\psi_1}{F_d(\mathrm{Ker}\,\psi_1)+
\mathrm{Im}\,\psi_2}\ar[r] &
0,
}
\]
we deduce that $\alpha_1$ induces an isomorphism 
\[
\frac{F_d(\mathrm{Ker}\,\phi_1)}{F_d(\mathrm{Im}\,\phi_2)}
\simeq 
\frac{F_d(\mathrm{Ker}\,\psi_1)}{F_d(\mathrm{Im}\,\psi_2)}.
\]
Similarly, from the diagram with exact rows 
\[
\xymatrix{
0 \ar[r] & \frac{F_d(\phi_{2}(M_{2}))}{\phi_{2}(F_d(M_{2}))}\ar[r] \ar[d] & 
\frac{F_d(\mathrm{Ker}\,\phi_1)}{\phi_2(F_d(M_2))}\ar[r]\ar[d]^{\mathrm{iso}} &
\frac{F_d(\mathrm{Ker}\,\phi_1)}{F_d(\phi_2(M_2))}\ar[r]\ar[d]^{\mathrm{iso}} & 0\\
0 \ar[r] & \frac{F_d(\psi_{2}(N_{2}))}{\psi_{2}(F_d(N_{2}))}\ar[r]& 
\frac{F_d(\mathrm{Ker}\,\psi_1)}{\psi_2(F_d(N_2))}\ar[r] &
\frac{F_d(\mathrm{Ker}\,\psi_1)}{F_d(\psi_2(N_2))}\ar[r] & 0,
}
\]
the result follows.
\end{proof}

\begin{corollary}\label{cor2}
Let $\alpha:M_{\bullet}\to N_{\bullet}$ be a morphism in $CF(\D)$ which is a quasi-isomorphism and which induces quasi-isomorphisms 
$F_d(M_{\bullet})\to F_d(N_{\bullet})$. Then $C_{\bullet}$ is strict if and only if $D_{\bullet}$ is strict.
\end{corollary}

Let $KF(\D)$ denote the category whose objects are the objects of $CF(\D)$ and the maps are taken modulo ($F$-adapted) homotopies. Then $DF(\D)$ denotes the localisation of $KF(\D)$ with respect to filtered quasi-isomorphisms, as done in \cite{saito88}, 2.1.8 (see also \cite{laumon}). In other terms, $DF(\D)$ is the localization of $KF(\D)$ with respect to the null system composed of the complexes $M_{\bullet}$ such that $\mathrm{gr}(M_{\bullet})$ is acyclic (see in \cite{KS} an introduction to localization of categories). $D(\D)$ will denote the derived category of $\D$-modules.

Because of Corollary \ref{cor2}, the strictness of a complex $M_{\bullet}$, where each $M_i$ is endowed with a good filtration, makes sense in $DF(\D)$. 

As pointed out in \cite{walther00}, for a strict complex $M_{\bullet}$, for any $i,d$, $H_i(F_d(M_{\bullet}))$ is a subspace of $H_i(\mathcal{M}_{\bullet})$,   
more precisely we can define a good filtration on $H_i(\mathcal{M}_{\bullet})$ by $F_d(H_i(\mathcal{M}))=H_i(F_d(M_{\bullet}))$.

Let $\mathfrak{m}$ denotes the maximal graded ideal of $\mathrm{gr}^F(\D)$ generated by 
\[
x_1,\dots,x_n,\xi_1,\dots,\xi_n
\]
 and $\mathfrak{m}^{(h)}$ the maximal graded two-sided ideal of $\D^{(h)}$ generated by 
 \[
 x_1,\dots,x_n,\partial_1,\dots,\partial_n,h.
 \] 
We have $\mathrm{gr}^F(\D)/\mathfrak{m}\simeq\C$ and $\D^{(h)}/\mathfrak{m}^{(h)}\simeq\C$.

\begin{definition}
Let $\mathcal{L}_{\bullet}=(\cdots\to \mathcal{L}_i\stackrel{\psi_i}{\to}\mathcal{L}_{i-1}\to\cdots)$ a complex of $CF(\D)$ with for each $i$, $\mathcal{L}_i=\D^{r_i}[\mathbf{n}^{(i)}]$. The complex $\mathcal{L}_{\bullet}$ is said to be minimal if for any $i$, the matrix representing $\mathbf{R}\,\psi_i$ as its entries in $\mathfrak{m}^{(h)}$ (equivalently, for any $i$, the matrix representing $\mathrm{gr}\,\psi_i$ as its entries in $\mathfrak{m}$).    
\end{definition} 	

\begin{definition}
Let $M_{\bullet}$ be a complex of $DF(\D)$. A minimal resolution of $M_{\bullet}$ is a minimal complex $\mathcal{L}_{\bullet}$ isomorphic to $M_{\bullet}$ in $DF(\D)$.
\end{definition}

The numbers $r^{(i)}$ and the shifts $\mathbf{n}^{(i)}$ arising in a minimal resolution of $M_{\bullet}$ make sense. In fact, denoting
\begin{eqnarray*}
\beta_{i,j} & = & 
\mathrm{dim}_{\C} \mathrm{Tor}^{\mathrm{gr}^F(\D)}_i(\mathrm{gr}^F(M),\C)_j\\
& = & 
\mathrm{dim}_{\C} \mathrm{Tor}^{\D^{(h)}}_i(\mathcal{R}_F(M),\C)_j,
\end{eqnarray*}
we have $\beta_{i,j}=\mathrm{card}\{k, n^{(i)}_k=j\}$. In particular $r_i=\sum_j \beta_{i,j}$.

Let us see finally the link between resolutions of modules and resolutions of complexes.
The notion of a strict filtered free resolution of a complex generalizes the notion of a filtered free resolution of a module: regarding a filtered module $M$ as a complex in $DF(\D)$ concentrated in degree $0$, a strict filtered free resolution $\cdots\to\mathcal{L}_1\to\mathcal{L}_0\to 0$ of $M$ (thus $H_0(\mathcal{L}_{\bullet})\simeq M$) provides a filtered free resolution  $\cdots\to\mathcal{L}_1\to\mathcal{L}_0\to H_0(\mathcal{L}_{\bullet})\to 0$ in the sense of \cite{OT01}.

\section{Filtered restriction}

Let $M$ be a $\D_{x,t}$-module endowed with a good bibiltration $(F_{d,k}(M))$ such that $M=\D_{x,t}V_0(M)$. We assume that there exists a non-zero polynomial $b(s)\in\C[s]$ such that $b(t_1\partial_{t_1}+\cdots+t_p\partial_{t_p})\mathrm{gr}^V_0(M)=0$, which is the case if $M$ is holonomic. An algorithm to compute such a polynomial can be found in \cite{OT01b}. Let $k_1$ be an integer such that $b(k)\neq 0$ if $k>k_1$. 

Let $i:(\C^n,0)\to (\C^n\times\C^p,0)$ denote the embedding $x\to(x,0)=(x,t)$. The restriction $i^{*}M$ of $M$ along $Y=\{t_1=\cdots=t_p=0\}$ is by definition the complex 
$(\D_{x,t}/\sum_i T_i\D_{x,t})\otimes^L_{\D_{x,t}}M$ in $D(\D)$-modules. Let $\Lambda^i=\Lambda^i\C^p$. Then $i^*M$ is represented by the Koszul complex over the sequence $t_1,\cdots,t_p$:
\[
0\to M\otimes\Lambda^p\stackrel{\delta}{\to}\cdots\to 
M\otimes\Lambda^{1}\stackrel{\delta}{\to}
M\otimes\Lambda^0\to 0.
\]
This complex is made of non finitely generated $\D$-modules. By    
\cite{oaku97}, Proposition 5.2 and \cite{OT01b}, section 5, the truncation 
\[
0\to V_{k_1+p}(M)\otimes\Lambda^p\stackrel{\delta}{\to}\cdots\to 
V_{k_1+1}(M)\otimes\Lambda^{1}\stackrel{\delta}{\to}
V_{k_1}(M)\otimes\Lambda^0\to 0,
\]
denoted by $V_{k_1+\bullet}\otimes\Lambda^{\bullet}$,
is quasi-isomorphic to the above Koszul complex, thus still represents $i^*M$. 
The bifiltration allows us to endow $V_{k_1+\bullet}\otimes\Lambda^{\bullet}$ with a filtration $F_d(V_{k_1+\bullet}\otimes\Lambda^{\bullet})$:
\[
0\to F_{d,k_1+p}(M)\otimes\Lambda^p\stackrel{\delta}{\to}\cdots\to 
F_{d,k_1+1}(M)\otimes\Lambda^{1}\stackrel{\delta}{\to}
F_{d,k_1}(M)\otimes\Lambda^0\to 0.
\]
Let us remark that since the bifiltration of $M$ is good, then there exists $f_1,\dots,f_r$ such that $F_{d,k}(M)=\sum_i F_{d-n_i,k-m_i}(\D)f_i$. If $k_1$ is chosen so that for any $i$, $k_1\geq m_i$, then each $f_i$ belongs to $V_{k_1}(M)$ and for each $i=0,\dots,p$, $F_{d,k_1+i}(M)$ is a good filtration of $V_{k_1+i}(M)$.

Let 
\[
\cdots\to\mathcal{L}_1\stackrel{\psi_1}{\to}\mathcal{L}_0\stackrel{\psi_0}{\to} M\to 0
\]
be any bifiltered free resolution of $M$, with 
$\mathcal{L}_i=\D_{x,t}^{r^{(i)}}[\mathbf{n}^{(i)}][\mathbf{m}^{(i)}]$. In particular it is a $V$-filtered free resolution.
By \cite{OT01b}, Theorem 5.3, the complex
\begin{equation*}
\cdots\to\frac{V_{k_1}(\mathcal{L}_1)}{\sum_i t_iV_{k_1+1}(\mathcal{L}_1)}
\to
\frac{V_{k_1}(\mathcal{L}_0)}{\sum_i t_iV_{k_1+1}(\mathcal{L}_0)}
\to 0,
\end{equation*}
denoted by 
$V_{k_1}(\mathcal{L}_{\bullet})/\sum_i t_iV_{k_1+1}(\mathcal{L}_{\bullet})$ is a free complex isomorphic to $i^*M$ in $D(\D)$. We endow it with a $F$-filtration by setting
\begin{eqnarray*}
F_d\left(\frac{V_{k_1}(\mathcal{L}_i)}{tV_{k_1+1}(\mathcal{L}_i)}\right)
 & = &\frac{F_{d,k_1}(\mathcal{L}_i)}{F_{d,k_1}(\mathcal{L}_i)\cap tV_{k_1+1}(\mathcal{L}_i)}\\
&\simeq &\frac{F_{d,k_1}(\mathcal{L}_i)}{tF_{d,k_1+1}(\mathcal{L}_i)}.
\end{eqnarray*}
Note that the filtered $\D$-module 
$V_{k_1}(\mathcal{L}_i)/\sum_i t_iV_{k_1+1}(\mathcal{L}_i)$
is naturally isomorphic to some $\D^{r}[\mathbf{n}]$.

\begin{proposition}\label{prop9}
The complexes $V_{k_1+\bullet}(M)\otimes\Lambda^{\bullet}$ and 
$V_{k_1}(\mathcal{L}_{\bullet})/\sum_i t_iV_{k_1+1}(\mathcal{L}_{\bullet})$ are isomorphic in $DF(\D)$. Thus $V_{k_1}(\mathcal{L}_{\bullet})/\sum_i t_iV_{k_1+1}(\mathcal{L}_{\bullet})$ is a filtered free resolution of $i^*M$.
\end{proposition}

\begin{proof}
Since $t_1,\dots,t_p$ is a regular sequence in $\mathcal{O}_{x,t}$, we have that the Koszul complex 
\[
0\to \mathbf{R}_{F,V}(\D_{x,t})\otimes\Lambda^p\stackrel{}{\to}\cdots\to 
\mathbf{R}_{F,V}(\D_{x,t})\otimes\Lambda^{0}\to 
\frac{\mathbf{R}_{F,V}(\D_{x,t})}{\sum_i t_i \mathbf{R}_{F,V}(\D_{x,t})}
\to 0
\]
is exact. We may replace $\D_{x,t}$ by the bifiltered free module $\mathcal{L}_i$, then the complex
\[
0\to \mathbf{R}_{F,V}(\mathcal{L}_i)\otimes\Lambda^p\stackrel{\delta}{\to}\cdots\to 
\mathbf{R}_{F,V}(\mathcal{L}_i)\otimes\Lambda^{0}\to 
\frac{\mathbf{R}_{F,V}(\mathcal{L}_i)}{\sum_i t_i \mathbf{R}_{F,V}(\mathcal{L}_i)}
\to 0
\]
is exact. Moreover it is bigraded, thus the complex
\[
0\to F_{d,k_1+p}(\mathcal{L}_i)\otimes\Lambda^p\stackrel{\delta}{\to}\cdots\to 
F_{d,k_1}(\mathcal{L}_i)\otimes\Lambda^{0}\to 
\frac{F_{d,k_1}(\mathcal{L}_i)}{\sum_i t_i F_{d,k_1+1}(\mathcal{L}_i)}\to 0
\]
is exact.
We have a commutative diagram:
\[
\xymatrix{
 & 0 & 0 & & \\
\cdots \ar[r] & F_{d,k_1+1}(M)\otimes\Lambda^1\ar[r]\ar[u] & 
F_{d,k_1}(M)\otimes\Lambda^0\ar[r]\ar[u] & 0 & \\
\cdots \ar[r] & F_{d,k_1+1}(\mathcal{L}_0)\otimes\Lambda^1\ar[r]\ar[u] & 
F_{d,k_1}(\mathcal{L}_0)\otimes\Lambda^0\ar[r]\ar[u] & 
\frac{F_{d,k_1}(\mathcal{L}_0)}{\sum_i t_i F_{d,k_1+1}(\mathcal{L}_0)}
\ar[r]\ar[u]
 & 0 \\
 \cdots\ar[r] & 
F_{d,k_1+1}(\mathcal{L}_1)\otimes\Lambda^1\ar[r]\ar[u] & 
F_{d,k_1}(\mathcal{L}_1)\otimes\Lambda^0\ar[r]\ar[u] & 
\frac{F_{d,k_1}(\mathcal{L}_1)}{\sum_i t_i F_{d,k_1+1}(\mathcal{L}_1)}
\ar[r]\ar[u]
 & 0 \\
 & \vdots\ar[u] & \vdots\ar[u] & \vdots\ar[u] & 
}
\]
We deduce that the complexes $(F_{d,k_1+i}(M)\otimes\Lambda^{i})_i$ and 
$(V_{k_1}(\mathcal{L}_{i})/\sum_i t_iV_{k_1+1}(\mathcal{L}_i))_i$ are both isomorphic to the complex associated with the double complex 
$$(F_{d,k_1+j}(\mathcal{L}_i)\otimes \Lambda^j)_{i,j}.$$
\end{proof}

\begin{corollary}\label{cor3}
The complex $V_{k_1}(\mathcal{L}_{\bullet})/\sum_i t_iV_{k_1+1}(\mathcal{L}_{\bullet})$ is strict if so is the complex 
$V_{k_1+\bullet}(M)\otimes\Lambda^{\bullet}.$
\end{corollary}

Let us consider from now on a special case: $p=1$. The complex $V_{k_1+\bullet}(M)\otimes\Lambda^{\bullet}$ is $V_{k_1+1}(M)\stackrel{t}{\to}V_{k_1}(M)$. Let us assume furthermore that $t:M\to M$ is injective. The restriction $i^*M$ is concentrated in degree $0$ with $H_0 i^*M\simeq M/tM$.
Thus the complex
\begin{equation}\label{eq1}
\cdots\to\frac{V_{k_1}(\mathcal{L}_1)}{tV_{k_1+1}(\mathcal{L}_1)}
\to
\frac{V_{k_1}(\mathcal{L}_0)}{tV_{k_1+1}(\mathcal{L}_0)}
\to \frac{M}{tM}\to 0
\end{equation}
is a free resolution of the module $M/tM=(V_{k_1}(M)+tM)/tM$. That module is naturally endowed with the filtration  
\[
F_d(M/tM):=\frac{F_{d,k_1}(M)}{F_{d,k_1}(M)\cap tM}.
\]
and from Proposition \ref{prop9} and Corollary \ref{cor3} we obtain: 

\begin{proposition}\label{thm1}
Assume that $t:M\to M$ is injective and that for any $d$, 
$F_{d,k_1}(M)\cap tV_{k_1+1}(M)=tF_{d,k_1+1}(M)$. Then the complex (\ref{eq1}) is an $F$-filtered free resolution of $M/tM$, i.e.\ for any $d$ the complex
\begin{equation}\label{eq3}
\cdots\to\frac{F_{d,k_1}(\mathcal{L}_1)}{tF_{d,k_1+1}(\mathcal{L}_1)}
\to
\frac{F_{d,k_1}(\mathcal{L}_0)}{tF_{d,k_1+1}(\mathcal{L}_0)}
\to F_d\left(\frac{M}{tM}\right)\to 0
\end{equation}
is exact. The homogenization, with respect to $F$, of the resolution (\ref{eq1}) is as follows:
\begin{equation}\label{}
\cdots\to\frac{V_{k_1}(\mathbf{R}\mathcal{L}_1)}{tV_{k_1+1}(\mathbf{R}\mathcal{L}_1)}
\to
 \frac{V_{k_1}(\mathbf{R}\mathcal{L}_0)}{tV_{k_1+1}(\mathbf{R}\mathcal{L}_0)}
\to \mathbf{R}\left(\frac{M}{tM}\right)\to 0
\end{equation}
\end{proposition}

Let us remark that the minimality of the bifiltered free resolution $\cdots\to\mathcal{L}_0\to M\to 0$ does not imply the minimality of the $F$-filtered free resolution (\ref{eq1}), as we shall see in Section 4.

Let us give suficient conditions for the assumnptions in \ref{thm1}.

\begin{proposition}\label{prop10}
The assumptions in Proposition \ref{thm1} hold if the map $t:\mathrm{gr}^F(M)\to \mathrm{gr}^F(M)$ is injective and for any $d,k$, $F_{d,k}(M)=F_d(M)\cap V_k(M)$.
\end{proposition}

That follows from the two following lemmas.

\begin{lemma}\label{lemma1}
Let us assume
\begin{enumerate}
\item The map $t:M\to M$ is injective,
\item For any $d$, $F_d(M)\cap tM=tF_d(M)$,
\item For any $d,k$, $F_{d,k}(M)=F_d(M)\cap V_k(M)$. 
\end{enumerate}
Then for any $d$, $tM\cap F_{d,k_1}(M)=  tF_{d,k_1+1}(M)$,
\end{lemma}

\begin{proof}
 By the injectivity of equation (5.2) in \cite{oaku97}, we have 
 $tM\cap V_{k_1}(M)=tV_{k_1+1}(M)$. Then 
\begin{eqnarray*}
tM\cap F_{d,k_1}(M) & \subset & tV_{k_1+1}(M)\cap F_{d}(M)\\
 & \subset &  tV_{k_1+1}(M)\cap tF_{d}(M)\quad\textrm{by assumption}\ 2\\
& \subset &  t(V_{k_1+1}(M)\cap F_{d}(M))\quad\textrm{by assumption}\ 1\\
& \subset &  tF_{d,k_1+1}(M)\quad\textrm{by assumption}\ 3.
\end{eqnarray*}
\end{proof}

\begin{lemma}
The conditions 1.\ and 2.\ of the preceding Lemma are satisfied if and only if the map $t:\mathrm{gr}^F(M)\to \mathrm{gr}^F(M)$ is injective. 
\end{lemma}



Let us give more information on the assumptions of Proposition \ref{prop10}.
 
\begin{lemma}[\cite{moi10}, Lemma 1.1]\label{lemma3}
$\forall d,k, F_{d,k}(M)=F_d(M)\cap V_k(M)$ holds if and only if the map induced by $h$ on $\mathrm{gr}^V(\mathbf{R}M)$ is injective.
\end{lemma}



\begin{remark}
The conditions in Proposition \ref{prop10} depend on the good bifiltration. The first condition is that $t:\mathrm{gr}^F(M)\to\mathrm{gr}^F(M)$ is injective. Let $M=\D_{x,t}^2/(\D_{x,t}.(t,1))$. If we define 
\[
F_d(M)=F_{d-1}(\D_{x,t}).\overline{(1,0)}+F_d(\D_{x,t}).\overline{(0,1)},	
\]
then $t$ is not injective on $\mathrm{gr}^F(M)\simeq\mathrm{gr}^F(\D_{x,t})^2/(t,0)$. 
On the other hand, defining
\[
F_d(M)=F_{d}(\D_{x,t}).\overline{(1,0)}+F_{d-1}(\D_{x,t}).\overline{(0,1)},	
\]
then $t$ is injective on $\mathrm{gr}^F(M)\simeq\mathrm{gr}^F(\D_{x,t})^2/(0,1)$.   

The second condition is that for any $d,k$, $F_{d,k}(M)=F_d(M)\cap V_k(M)$, which is equivalent to the injectivity of the map induced by $h$ on 
$\mathrm{gr}^V(\mathbf{R}(M))$. Let us take $M=\D_{x,t}$ endowed with the usual bifiltration, then the above condition is satisfied. Let $M'=\D_{x,t}^2/(1,1)$. We have an isomorphism $M\simeq M'$ given by $1\mapsto \overline{(1,0)}$. Let us endow $M'$ with the good bifiltration defined by
\[
F_{d,k}(M')=F_{d-1,k}(\D_{x,t}).\overline{(1,0)}+
F_{d,k-1}(\D_{x,t}).\overline{(0,1)}.
\] 
Then $\mathbf{R}(M)\simeq (\D_{x,t}^{(h)})^2[0,1]/(1,h)$
and $\mathrm{gr}^V(\mathbf{R}(M))\simeq 
(\mathrm{gr}^V(\D_{x,t}^{(h)}))^2/(0,h))$, on which the map induced by $h$ is not injective. 
\end{remark}

To end this section, let us recall the notion of involutive bases, which we will use in Section 4.

\begin{definition}
Let $M$ be a filtered $\D$-module. If $0\neq m\in M$, we define the $F$-order of $m$ by
$\mathrm{ord}^F(m)=\mathrm{min} \{d, m\in F_d(M)\}$
and the $F$-symbol $\sigma^F(m)$ by the class of $m$ in $\mathrm{gr}_{\mathrm{ord}^F(m)}^F(M)$. If it is clear in the context we will simply note $\sigma(m)$.

Let $I$ be an ideal of $\D$, endowed with the induced filtration.
We say that $(P_1,\dots,P_r)$ is an $F$-involutive base of $I$ if for any $P\in I$, there exist $Q_1,\dots,Q_r$ such that $P=\sum Q_iP_i$ and for any $i$, $\mathrm{ord}^F(Q_iP_i)\leq \mathrm{ord}^F(P)$.
\end{definition}

We have the following useful characterization of involutive bases. 

\begin{proposition}\label{prop5}
Let $I$ be generated by $P_1,\dots,P_r$ and let $n_i=\mathrm{ord}^F(P_i)$. Assume that we have homogeneous elements 
\[
S_i=\sum_j S_{i,j}e_j\in\mathrm{gr}^F(\D)^r[\mathbf{n}]
\]
 which generate the relations between $\sigma(P_1),\dots,\sigma(P_r)$, such that for any $i$, there exists a relation 
 \[
 R_i=\sum_j R_{i,j}e_j\in \D[\mathbf{n}]
 \]
  between $P_1,\dots,P_r$, such that $\sigma(R_i)=S_i$. Then $(P_1,\dots,P_r)$ is an $F$-involutive base of $I$.
\end{proposition}

\begin{proof}
Let $P=\sum Q_jP_j$ and assume that $\mathrm{ord}(P)<\mathrm{ord}(\sum Q_je_j)=d$ with $d$ minimal. Then $\sigma(\sum Q_jP_j)=0$, which leads to a relation $\sigma(\sum Q_je_j)$ between the ($\sigma(P_j))$. Thus for any $i$ there exists $a_i=\sigma(b_i)$ such that  $\sigma(\sum Q_je_j)=\sum a_iS_i=\sigma(\sum b_iR_i)$.
Then 
\[
P=(\sum Q_je_j-\sum b_iR_i).(P_1,\dots,P_r)
\]
 with $\mathrm{ord}(\sum Q_je_j-\sum b_iR_i)<d$, a contradiction.
\end{proof}

We will say that the relation $S_i$ is \emph{lifted} by the relation $R_i$ such that $\sigma(R_i)=S_i$. 
In fact it is easy to see that the converse of the lemma holds: if $P_1,\dots,P_r$ is an $F$-involutive base, then every homogeneous relation between $(\sigma(P_i))$ can be lifted by a relation between the $(P_i)$.

\section{Application to algebraic local cohomology}

Our aim is to give presentations of some algebraic local cohomology module, viewed as a $\D$-module. Let $\mathcal{O}=\C\{\mathbf{x}\}$ and let $f\in\mathcal{O}$. We focus on the module $N=\mathcal{O}[1/f]/\mathcal{O}$ which is quasi-isomorphic to the complex of algebraic local cohomology 
$\mathbb{R}^{\bullet+1}\Gamma_{[f=0]}(\mathcal{O})$. Although this module is not finitely generated over $\mathcal{O}$, it turns to be a finitely generated $\D$-module. We will assume that $f$ is quasi-homogeneous and has an isolated singularity at the origin. We will give two minimal presentations of $N$, one of which being classical, the other coming from our general result in Section 3.

Let $f\neq 0$ be any function in $\mathcal{O}$. The vector space $\mathcal{O}[1/f,s]f^s$, where $f^s$ is understood as a symbol, has a natural structure of a $\D[s]$-module. Let us recall that there exists a polynomial $b_f(s)\in\C[s]$ called \emph{Bernstein-Sato polynomial} such that there exists $P(s)\in \D[s]$ such that 
\begin{equation}\label{eq7}
b(s)f^s=P(s)f^{s+1}.
\end{equation}   
Let $-k'$ be the least integral root of $b(s)$. It is known that $k'\geq 1$. Using the functional equation (\ref{eq7}), one has 
$
\mathcal{O}[1/f]=\mathcal{D}(1/f^{k'}).
$ 
Let us assume from now on that $f$ has an isolated singularity at the origin. Let 
\[
S_{i,j}=f'_i\partial_j-f'_j\partial_i
\]
for any $0<i<j\leq n$. Then the symbols $(\sigma^F(S_{i,j}))$ generate the kernel of the map of $\mathcal{O}$-algebras
\[
\phi_f:\mathcal{O}[\mathbf{\xi}]\to\bigoplus (J(f))^dT^d
\]
defined by $\xi_i\mapsto f'_iT$, and the operators $(S_{i,j})$ generate $\mathrm{Ann}_{\D}f^s$, see \cite{narvaez02}. Let us identify $\mathrm{gr}^F(\D)$ with $\C\{x_1,\dots,x_n\}[\xi_1,\dots,\xi_n]$. It is also known that the relations between the symbols $(\sigma(S_{i,j}))$ are generated by
\begin{equation}\label{eq8}
f'_i\sigma(S_{j,k})-f'_j\sigma(S_{i,k})+f'_k\sigma(S_{i,j})=0
\end{equation}
and
\begin{equation}\label{eq9}
\xi_i\sigma(S_{j,k})-\xi_j\sigma(S_{i,k})+\xi_k\sigma(S_{i,j})=0
\end{equation}
for each triple $i<j<k$ (see \cite{moi10}, paragraph 3.1.1). These relations are lifted by the following relations in $\D$:
\begin{equation*}
f'_i S_{j,k}-f'_j S_{i,k}+f'_k S_{i,j}=0
\end{equation*}
and
\begin{equation*}
\partial_i S_{j,k}-\partial_j S_{i,k}+ \partial_k S_{i,j}=0.
\end{equation*}

We also assume that $f$ is quasi-homogeneous, i.e.\ there exist positive weights $w_1,\dots,w_n$ such that, denoting $\theta=\sum w_i x_i\partial_i$, we have $\theta(f)=f$. 

\subsection{Classical presentation of  
$\mathcal{O}[1/f]/\mathcal{O}$}

First let us consider the module $\mathcal{O}[1/f]$ endowed with the following good $F$-filtration:
\[
F_d\left(\mathcal{O}\left[\frac{1}{f}\right]\right)=F_d(\mathcal{D})\frac{1}{f^{k'}}.
\] 

\begin{proposition}\label{prop1}
The module $\mathcal{O}[1/f]$ admits the minimal filtered presentation 
$
\mathcal{L}_1\stackrel{\phi_1}{\to}
\D[0]\stackrel{\phi_0}{\to} \mathcal{O}[1/f]\to 0
$
where 
\begin{itemize}
\item $\phi_0(1)=1/f^{k'}$, 
\item $\mathcal{L}_1=\D e\oplus(\bigoplus_{i<j}\D e_{i,j})$ with 
$\mathrm{ord}(e)=1$ and $\mathrm{ord}(e_{i,j})=1$,
\item $\phi_1(e)=\theta+k'$,
\item $\phi_1(e_{i,j})=S_{i,j}$.
\end{itemize}
\end{proposition}

\begin{proof}
Since $\mathrm{Ann}_{\D}f^s$ is generated by $(S_{i,j})$, we have that $\mathrm{Ann}_{\D[s]}f^s$ is generated by $(S_{i,j})$ and $s-\theta$. Using \cite{cimpa1}, Remark 14, we deduce that $\mathrm{Ann}_{\D}(1/f^{k'})$ is generated by $(S_{i,j})$ and $\theta+k$. Let us show that this system of generators is $F$-involutive.

We claim that $\sigma(\theta)$ is a non-zero divisor on $\mathcal{O}[\mathbf{\xi}]/(\sigma(S_{i,j}))$. Indeed, assume that 
$P\sigma(\theta)\in (S_{i,j})$. Then 
$0=\phi_f(P\sigma(\theta))=\phi_f(P)f$
thus $\phi_f(P)=0$ and $P\in (\sigma(S_{i,j}))$ as claimed. Next, the relation
\begin{equation}\label{eq11}
\sigma(\theta)\sigma(S_{i,j})-\sigma(S_{i,j})\sigma(\theta)=0
\end{equation}
can be lifted since 
$
[\theta+k',S_{i,j}]=(1-w_i-w_j)S_{i,j}
$
which reads
\[
(\theta+k'-1+w_i+w_j)S_{i,j}-S_{i,j}(\theta+k')=0.
\]
The involutivity follows then from Proposition \ref{prop5}, and the minimality of the presentation is clear. 
\end{proof}

Now we turn to the module $N$, which we endow similarly with the following good $F$-filtration:
\[
F_d(N)=F_d(\mathcal{D})\left[\frac{1}{f^{k'}}\right].
\] 

\begin{proposition}\label{prop2}
The module $N=\mathcal{O}[1/f]/\mathcal{O}$ admits the minimal filtered presentation
$
\mathcal{L}_1
\stackrel{\phi_1}{\to}\D[0]\stackrel{\phi_0}{\to} N\to 0
$
where 
\begin{itemize}
\item $\phi_0(1)=[1/f^{k'}]$, 
\item $\mathcal{L}_1=\D e_0\oplus\D e_1\oplus(\bigoplus_{i<j}\D e_{i,j})$ with 
$\mathrm{ord}(e_0)=0$, $\mathrm{ord}(e_1)=1$, $\mathrm{ord}(e_{i,j})=1$, 
\item $\phi_1(e_0)=f^{k'}$,
\item $\phi_1(e_1)=\theta+k'$,
\item $\phi_1(e_{i,j})=S_{i,j}$.
\end{itemize}
\end{proposition}

\begin{proof}
First, let us show that $\mathrm{Ker}\phi_0=\mathrm{Im}\phi_1$. For $P\in\D$, $\phi_0(P)=0$ if and only if 
$P.(1/f^{k'})=g\in\mathcal{O}$. 
That reads $(P-gf^{k'}).(1/f^{k'})=0$ and the result follows from Proposition \ref{prop1}.

Let us show that $(f^{k'},\theta+k',(S_{i,j}))$ form an $F$-involutive basis.
We claim that the relations between the symbols $\sigma(f^{k'})=f^{k'}$, $\sigma(\theta+k')=\sigma(\theta)$, $(\sigma(S_{i,j}))$ are generated by the relations (\ref{eq8}), (\ref{eq9}), (\ref{eq11}) and the following:
\begin{equation}\label{eq10}
-\xi_jf^{k'}+f^{k'-1}f'_j\sigma(\theta)-\sum_i f^{k'-1}w_ix_i\sigma(S_{j,i})=0
\end{equation}
for any $j=1,\dots,n$. 
Let $P\in\mathrm{gr}^F(\D)$ homogeneous such that 
\[
Pf^{k'}\in ((\sigma(\theta)),(\sigma(S_{i,j}))).
\]
 The element $P$ must have degree at least $1$, i.e. $P\in (\xi_i)$ and conversely each $P\in(\xi_i)$ is allowed, using $(\ref{eq10})$. The relations between the elements $(\theta+k,(S_{i,j}))$
are generated by $(\ref{eq8}), (\ref{eq9}), (\ref{eq11})$, which proves our claim. We then conclude to prove that $(f^k,\theta+k,(S_{i,j}))$ form an $F$-involutive basis by saying that the relations given above can be lifted: we already mentioned it for $(\ref{eq8}), (\ref{eq9}), (\ref{eq11})$; and (\ref{eq10}) is lifted as follows:
\[
-\partial_jf^{k'}+f^{k'-1}f'_j(\theta+k')-
\sum_i f^{k'-1}w_ix_iS_{j,i}=0.
\]
Finally, these computations prove that the presentation is minimal.
\end{proof}

\subsection{Another presentation of $\mathcal{O}[1/f]/\mathcal{O}$}

We still assume that $f$ is quasi-homogeneous with an isolated singularity at the origin, and we will derive from Proposition \ref{thm1} and our results in \cite{moi10} another minimal presentation of the module 
$N=\mathcal{O}[1/f]/\mathcal{O}$. Let $\mathcal{O}_{x,t}=\C\{\mathbf{x},t\}$. We will apply Proposition \ref{thm1} to the $\D_{x,t}$-module
\[
M=\frac{\mathcal{O}_{x,t}[\frac{1}{f-t}]}{\mathcal{O}_{x,t}}.
\] 
which is quasi-isomorphic to the algebraic local cohomology 
$\mathbb{R}^{\bullet+1}\Gamma_{[f-t=0]}(\mathcal{O}_{x,t})$. There is a $\D_{x,t}$-module structure on $\mathcal{O}[1/f,s]f^s$ such that $s$ acts as $-\partial_t t$ and 
$
\D_{x,t}f^s\simeq M
$
by the mapping $f^s\mapsto [1/(f-t)]$. We endow $M$ with the good bifiltration 
\[
F_{d,k}(M)=F_{d,k}(\D_{x,t})[1/(f-t)].
\]

\begin{lemma}
$M$ satisfies the assumptions in Proposition \ref{prop10}:
\begin{enumerate}
\item $t:\mathrm{gr}^F(M)\to\mathrm{gr}^F(M)$ is injective,
\item For any $d,k$, $F_{d,k}(M)=F_d(M)\cap V_k(M)$.
\end{enumerate}
\end{lemma}

\begin{proof}
By e.g.\ \cite{moi10}, Lemma 2.1., we have 
\[
\mathrm{gr}^F(M)\simeq\frac{\mathrm{gr}^F(\D_{x,t})}{(t-f,(\xi_i+f'_i\tau)_i)},
\]
then 1.\ comes from the fact that $(t,t-f,(\xi_i+f'_i\tau)_i))$ form a regular sequence. Finally, \cite{moi10}, Proposition 3.2, Proposition 2.6 and Lemma 1.1 give 2.
\end{proof}

\begin{lemma}\label{lemma2}
There is an isomorphism $M/tM\simeq N$
defined as follows: if $g(x,t)\in\mathcal{O}_{x,t}$, then $[g(x,t)/(f-t)^r]$ is mapped to $[g(x,0)/f^r]$.
\end{lemma}

The proof is straightforward. Let us note, for example, that the image of $\partial_t^i.[1/(f-t)]$ is $(j!)[1/f^{j+1}]$.

The $b$-function of $M$ as defined by Oaku-Takayama, denoted by $b_M(X)$, and the Bernstein-Sato polynomial $b_f(s)$ of $f$ are very close. By definition, $b_f(s)$ satisfies
$$
b_f(s)\D[s]f^s\subset\D[s]f^{s+1}.
$$ 
On the other hand, we have $V_0(M)=\D[s]f^s$ and 
$V_{-1}(M)=\D[s]f^{s+1},$ thus $b_M(X)$ is defined by 
$$
b_M(t\partial_t)\D[s]f^s\subset\D[s]f^{s+1}.
$$
Identifying  $s$ and $-\partial_tt$, we get $b_M(-(s+1))=b_f(s)$, and finally $k'=k_1+1.$

From now on we endow the module $N$ with the following good $F$-filtration, which takes into account the order of the pole $f$:
$$
F_d(N)=\sum_{1\leq j\leq k'}F_{d-(j-1)}(\D)
\left[\frac{1}{f^{j}}\right],
$$ 
i.e.\ $\mathrm{ord}^F(1/f^j)=j-1$ for any $1\leq j\leq k'$.
In particular, the identity $1/f^j=f.1/f^{j+1}$ for $1\leq j\leq k'-1$ is not adapted to that filtration.

\begin{proposition}\label{prop4}
The module $N=\mathcal{O}[1/f]/\mathcal{O}$ admits the minimal filtered presentation
$
\mathcal{L}_1
\stackrel{\varphi_1}{\to}
\mathcal{L}_0
 \stackrel{\varphi_0}{\to} N\to 0
$
where 
\begin{description}
\item $\mathcal{L}_0=\bigoplus_{k=0}^{k_1}\D\partial_t^k$
with $\mathrm{ord}^F(\partial_t^k)=k$,
\item $\varphi_0(\partial_t^k)=(k!)[1/f^{k+1}]$,
\item $\mathcal{L}_1=
(\bigoplus_{k=0}^{k_1}\D\partial_t^kX_1)
\oplus(\bigoplus_{i=1}^n\bigoplus_{k=0}^{k_1-1}\D\partial_t^ke_i)
\oplus \D X_2$ 
\item with $\mathrm{ord}^F(\partial_t^kX_1)=k+1$, 
$\mathrm{ord}^F(\partial_t^ke_i)=k+1$,
$\mathrm{ord}^F(X_2)=0$,
\item $\varphi_1(\partial_t^kX_1)=(k+1+\theta)\partial_t^k$,
\item $\varphi_1(\partial_t^ke_i)=
-(f'_i\partial_t^{k+1}+\partial_i\partial_t^k)$,
\item $\varphi_1(X_2)=f$. 
\end{description}
\end{proposition}

\begin{proof}
Let $S_1=\C X_1\oplus \C X_2$ and $\Lambda^j=\Lambda^j\C^n$ with basis $e_{i_1}\wedge\cdots\wedge e_{i_j}$ with $1\leq i_1<\cdots<i_j\leq n$. In the sequel, the tensor products are understood over $\C$.  Let $\delta_1$ (resp.\ $\delta_2$) denote the Koszul map associated with the sequence $f'_1,\dots,f_n$ (resp.\ $-\partial_1,\dots,-\partial_n$). For $i=1,2$, $\delta_i$ is a map
$\D\otimes\Lambda^{\bullet}\to\D\otimes\Lambda^{\bullet-1}$ and can be extended to a map $\D_{x,t}\otimes\Lambda^{\bullet}\to\D_{x,t}\otimes\Lambda^{\bullet-1}$ or to a map $\D^{(h)}\partial_t^k\otimes\Lambda^{\bullet}\to
\D^{(h)}\partial_t^k\otimes\Lambda^{\bullet-1}$ as well.

Our starting point is the beginning of the bifiltered free resolution of $M$, lifted from the bigraded resolution  computed in \cite{moi10}, proof of Theorem 3.2:
\[
\mathcal{L}_2\stackrel{\psi_2}{\to}
\mathcal{L}_1\stackrel{\psi_1}{\to}
\D_{x,t}[0][0]\stackrel{\psi_0}{\to} M\to 0
\]
where
\begin{description}
\item $\psi_0(1)=[1/(f-t)]$;
\item $\mathcal{L}_1=(\D_{x,t}\otimes S_1)\oplus 
(\D_{x,t}\otimes \Lambda^1)\oplus(\D_{x,t}\otimes \Lambda^2)$
\item with $\mathrm{ord}^{F,V}(X_1)=(1,0)$;
$\mathrm{ord}^{F,V}(X_2)=(0,0)$;
$\mathrm{ord}^{F,V}(e_i)=(1,1)$;\\
$\mathrm{ord}^{F,V}(e_i\wedge e_j)=(1,0)$;
\item $\psi_1(X_1)=\partial_t t+\theta$; $\psi_1(X_2)=f-t$;
$\psi_1(e_i)=-f'_i\partial_t -\partial_i$;
$\psi_1(e_i\wedge e_j)=-S_{i,j}$;
\item  $\mathcal{L}_2=\D_{x,t}.1\oplus (\D_{x,t}\otimes \Lambda^1)\oplus 
(\D_{x,t}\otimes \Lambda^2\otimes S_1)\oplus(\D_{x,t}\otimes \Lambda^3\otimes S_1)$
\item with $\mathrm{ord}^{F,V}(1)=(1,1)$;
$\mathrm{ord}^{F,V}(e_i)=(1,0)$;
$\mathrm{ord}^{F,V}(e_i\wedge e_j\otimes X_1)=(1,1)$;
$\mathrm{ord}^{F,V}(e_i\wedge e_j\otimes X_2)=(2,1)$;
$\mathrm{ord}^{F,V}(e_i\wedge e_j\wedge e_k\otimes X_1)=(1,0)$;
$\mathrm{ord}^{F,V}(e_i\wedge e_j\wedge e_k\otimes X_2)=(2,0)$;
\item $\psi_2(1)=X_1+\partial_t X_2+\sum w_jx_je_j$;
$\psi_2(e_i)=f'_iX_1-\partial_iX_2+te_i+\sum_j w_jx_je_i\wedge e_j$;
\item $\psi_2(e_i\wedge e_j\otimes X_1)=\delta_1(e_i\wedge e_j)+e_i\wedge e_j$;
$\psi_2(e_i\wedge e_j\otimes X_2)=\delta_2(e_i\wedge e_j)-\partial_te_i\wedge e_j$;
\item for $I=(i,j,k)$, $\psi_2(e_I\otimes X_l)=\delta_l(e_I)$.
\end{description}

By homogenizing with respect to $F$, this gives the beginning of a minimal $V$-adapted free resolution  
\[
\mathbf{R}\mathcal{L}_2\stackrel{\mathbf{R}\psi_2}{\to}
\mathbf{R}\mathcal{L}_1\stackrel{\mathbf{R}\psi_1}{\to}
\D_{x,t}^{(h)}[0][0]\stackrel{\mathbf{R}\psi_0}{\to} \mathbf{R}M\to 0
\]
with the maps $\mathbf{R}\psi_i$ for $i=1,2$ having the same expressions as the maps $\psi_i$.

Then we apply Proposition \ref{thm1}, we obtain the following exact sequence:
\[
\frac{V_{k_1}(\mathbf{R}\mathcal{L}_2)}{tV_{k_1+1}(\mathbf{R}\mathcal{L}_2)}
\stackrel{\overline{\mathbf{R}\psi_2}}{\to}
\frac{V_{k_1}(\mathbf{R}\mathcal{L}_1)}{tV_{k_1+1}(\mathbf{R}\mathcal{L}_1)}
\stackrel{\overline{\mathbf{R}\psi_1}}{\to}
 \frac{V_{k_1}(\D_{x,t}^{(h)})}{tV_{k_1+1}(\D_{x,t}^{(h)})}
\stackrel{\overline{\mathbf{R}\psi_0}}{\to}
\mathbf{R}N\to 0.
\]
We have
\[
\mathcal{F}_0:=
\frac{V_{k_1}(\D_{x,t}^{(h)})}{tV_{k_1+1}(\D_{x,t}^{(h)})}
\simeq \bigoplus_{k=0}^{k_1}\D^{(h)}\partial_t^k
\]
with $\mathrm{deg}(\partial_t^k)=k$;
\[
\mathcal{F}_1:=
\frac{V_{k_1}(\mathbf{R}\mathcal{L}_1)}{tV_{k_1+1}(\mathbf{R}\mathcal{L}_1)}
\simeq \bigoplus_{k=0}^{k_1}\D^{(h)}\partial_t^k\otimes S_1
\oplus (\bigoplus_{k=0}^{k_1-1}\D^{(h)}\partial_t^k\otimes \Lambda^1)
\oplus (\bigoplus_{k=0}^{k_1}\D^{(h)}\partial_t^k\otimes \Lambda^2)
\]
with $\mathrm{deg}(\partial_t^k\otimes X_1)=k+1$;
$\mathrm{deg}(\partial_t^k\otimes X_2)=k+1$;
$\mathrm{deg}(\partial_t^k\otimes e_i)=k+1$;
$\mathrm{deg}(\partial_t^k\otimes e_i\wedge e_j)=k+1$;
\begin{eqnarray*}
\mathcal{F}_2 & := &
\frac{V_{k_1}(\mathbf{R}\mathcal{L}_2)}{tV_{k_1+1}(\mathbf{R}\mathcal{L}_2)} \\
 & \simeq & (\bigoplus_{k=0}^{k_1-1}\D^{(h)}\partial_t^k)
\oplus (\bigoplus_{k=0}^{k_1}\D^{(h)}\partial_t^k\otimes \Lambda^1)
\oplus (\bigoplus_{k=0}^{k_1-1}\D^{(h)}\partial_t^k\otimes \Lambda^2\otimes S_1)\\
 & \oplus & (\bigoplus_{k=0}^{k_1}\D^{(h)}\partial_t^k\otimes \Lambda^3\otimes S_1)
\end{eqnarray*}
with $\mathrm{deg}(\partial_t^k.1)=k+1$;
$\mathrm{deg}(\partial_t^k\otimes e_i)=k+1$;
$\mathrm{deg}(\partial_t^k\otimes e_i\wedge e_j\otimes X_1)=k+1$;
$\mathrm{deg}(\partial_t^k\otimes e_i\wedge e_j\otimes X_2)=k+2$;
for $I=(i,j,k), \mathrm{deg}(\partial_t^k\otimes e_I\otimes X_1)=k+1$;
$\mathrm{deg}(\partial_t^k\otimes e_I\otimes X_2)=k+2$.

Let us compute the maps $\overline{\mathbf{R}\psi_i}$. 
We have $\overline{\mathbf{R}\psi_0}(\partial_t^k)=(k!)[1/f^{k+1}]T^k$.
In $\D_{x,t}^{(h)}$ the following identity holds:
\[
\partial_t^kt=t\partial_t^k+kh\partial_t^{k-1}.
\]
Then
\begin{eqnarray*}
\overline{\mathbf{R}\psi_1}(\partial_t^k\otimes X_1) & = &
[\partial_t^k(\partial_t+\theta)]\\
 & = & [t\partial_t^{k+1}t+((k+1)h+\theta)\partial_t^k]\\
 & = & ((k+1)h+\theta)\partial_t^k.
\end{eqnarray*}
In the same way the computation gives:
\begin{description}
\item $\overline{\mathbf{R}\psi_1}(\partial_t^k\otimes X_2)
=f\partial_t^k-kh\partial_t^{k-1}$
\item $\overline{\mathbf{R}\psi_1}(\partial_t^k\otimes e_i)
=-(f'_i\partial_t^{k+1}+\partial_i\partial_t^k)$
\item $\overline{\mathbf{R}\psi_1}(\partial_t^k\otimes e_i\wedge e_j)
=-S_{i,j}\partial_t^k$
\item $\overline{\mathbf{R}\psi_2}(\partial_t^k.1)
=\partial_t^k\otimes X_1+\partial_t^{k+1}\otimes X_2+
(\sum_j w_jx_j\partial_t^k\otimes e_j)$
\item $\overline{\mathbf{R}\psi_2}(\partial_t^k\otimes e_i)
=f'_i\partial_t^k\otimes X_1-\partial_i\partial_t^k\otimes X_2+
kh\partial_t^{k-1}\otimes e_i+\sum_jw_jx_j\partial_t^k\otimes e_i\wedge e_j$
\item $\overline{\mathbf{R}\psi_2}(\partial_t^k\otimes e_i\wedge e_j\otimes X_1)
=\delta_1(\partial_t^k\otimes e_i\wedge e_j)+\partial_t^k e_i\wedge e_j$
\item $\overline{\mathbf{R}\psi_2}(\partial_t^k\otimes e_i\wedge e_j\otimes X_2)
=\delta_2(\partial_t^k\otimes e_i\wedge e_j)-\partial_t^{k+1} e_i\wedge e_j$
\item $\overline{\mathbf{R}\psi_2}(\partial_t^k\otimes e_I\otimes X_l)
=\delta_l(\partial_t^k\otimes e_I)$.
\end{description}   

Since some of the entries of the matrix which represents  $\overline{\mathbf{R}\psi_2}$ are units, we have to minimalize. Our method is the same as the minimalization process described in \cite{GO04}, paragraph 4.4.
Let 
\[
\mathcal{G}=\bigoplus_{k=0}^{k_1-1}\D^{(h)}\partial_t^k\oplus
\bigoplus_{i<j}\bigoplus_{k=0}^{k_1-1}\D^{(h)}\partial_t^k\otimes e_i\wedge e_j\otimes X_2\oplus
\bigoplus_{i<j}\D^{(h)}\otimes e_i\wedge e_j\otimes X_1,
\]
a free submodule of $\mathcal{F}_2$. We have the following commutative diagram:
\[
\xymatrix{
0 \ar[d] &  0\ar[d]  & \\ 
\mathcal{G} \ar[d] \ar[r]^{\overline{\mathbf{R}\psi_2}} & 
\overline{\mathbf{R}\psi_2}(\mathcal{G})\ar[d] 
\ar[r] & 0\ar[d]\\
\mathcal{F}_2\ar[r]^{\overline{\mathbf{R}\psi_2}}\ar[d] &
\mathcal{F}_1 \ar[r]^{\overline{\mathbf{R}\psi_1}}\ar[d] & 
\mathcal{F}_0\ar[d]^{\mathrm{Id}}\\
\frac{\mathcal{F}_2}{\mathcal{G}} \ar[r]^{\phi_2} \ar[d] &
\frac{\mathcal{F}_1}{\overline{\mathbf{R}\psi_2}(\mathcal{G})}
\ar[r]^{\phi_1}\ar[d] & 
\mathcal{F}_0\ar[d]\\
0 & 0 &   0, 
}
\]
where $\phi_i$ is the map induced by $\overline{\mathbf{R}\psi_i}$ for $i=1,2$.
The columns are exact and the first two rows are exact. Then we deduce that the third row is also exact. This row will provide our minimal presentation. 
The module 
$\mathcal{F}'_2:=\mathcal{F}_2/\mathcal{G}$ is free with basis $(\partial_t^k\otimes e_i)_{i,k}$, $(\partial_t^k\otimes e_i\wedge e_j\otimes X_1)_{k\geq 1, i, j}$ and $(\partial_t^k\otimes e_I\otimes X_j)_{I=(i_1,i_2,i_3),k,j=1,2}$. It is easy to show that
\[
\mathcal{F}'_1:= 
\frac{\mathcal{F}_1}{\overline{\mathbf{R}\psi_2}(\mathcal{G})}
\]
is free with basis $(\partial_t^k\otimes X_1)_k$, $(\partial_t^k\otimes e_i)_{i,k}$ and $X_2$.
Note that in $\mathcal{F}'_1$, we have 
\begin{eqnarray*}
e_i\wedge e_j & \equiv & -\delta_1(e_i\wedge e_j),\\
\partial_t^{k}\otimes e_i\wedge e_j & \equiv & 
\delta_2(\partial_t^{k-1}\otimes e_i\wedge e_j)\quad\textrm{for}\ k\geq 1,\\
\partial_t^k\otimes X_2 & \equiv & -(\partial_t^{k-1}\otimes X_1+
\sum w_jx_je_j\partial_t^{k-1})\quad\textrm{for}\ k\geq 1.  
\end{eqnarray*}
Using these identities, we compute the matrix representing $\phi_2$ in the basis decribed above:
 
\noindent For $k\geq 1$, 
\[
\phi_2(\partial_t^k\otimes e_i)=
f'_1\partial_t^k\otimes X_1+\partial_i\partial_t^{k-1}\otimes X_1+
(k+w_i)h\partial_t^{k-1}\otimes e_i+\sum_j w_j x_j\partial_j\partial_t^{k-1}\otimes e_i;
\]
\begin{description}
\item  $\phi_2(e_i)=f'_iX_1-\partial_iX_2-\sum_j w_jx_j\delta_1(e_i\wedge e_j)$;
\item For $k\geq 1$, $\phi_2(\partial_t^k\otimes e_i\wedge e_j\otimes X_1)=
\delta_1(\partial_t^k\otimes e_i\wedge e_j)+
\delta_2(\partial_t^{k-1}\otimes e_i\wedge e_j)$.
\end{description}
Let $I=(i_1,i_2,i_3)$.
\begin{description}
\item For $k\geq 1$,
 $\phi_2(\partial_t^k\otimes e_I\otimes X_1)=
 \delta_2(\delta_1(\partial_t^{k-1}\otimes e_I))$;
 \item $\phi_2(e_I\otimes X_1)=-\delta_1(\delta_1(e_I))=0$;
 \item For $k\geq 1$,
 $\phi_2(\partial_t^k\otimes e_I\otimes X_2)=0$;
 \item $\phi_2(e_I\otimes X_2)=-\delta_1(\delta_2(e_I))$.
\end{description}
Thus the matrix representing $\phi_2$ does not contain any unity. By dehomogenizing we obtain the annunciated minimal presentation of $N$.
\end{proof}

\subsection*{Acknowledgements}
I sincerely thank Toshinori Oaku for having suggested the problem discussed in section 3, and for his comments. I am also grateful to Michel Granger for his advices throughout the elaboration of that paper. The project was funded by the Japan Society for the Promotion of Science, The Institute of Mathematics of the University of Seville IMUS, 
the University of Seville (Spain) and the Erwin Schr\"odinger International Institute for Mathematical Physics (Vienna, Austria).

\end{document}